\documentclass[12pt]{article}
\usepackage{amssymb, amsmath}
\author{Carlos Curr\'{a}s-Bosch  \footnote{partially supported by  DGICYT: MTM 2006-04353.} }

\begin{document}

\title {\bf   Singular cotangent model}

\maketitle\frenchspacing

\newtheorem{Nota}{ Notation}[section]

\newtheorem{defi}{ Definition}[section]
\newtheorem{lemma}{ Lemma}[section]
\newtheorem{proposition}{Proposition}[section]

\newtheorem{claim}{ Claim}[section]
\newtheorem{remark} {Remark} [section]
\newtheorem{obs}{ Observation}[section]
\newtheorem{cor} { Corollary}[section]
\newtheorem{Th} {Theorem}[section]
\newenvironment{proof}{\noindent {\bf Proof:}}{\quad \hfill $\square$}
\begin{abstract}
Any singular level of a completely integrable system (c.i.s.) with non-degenerate singularities has a singular affine structure. We shall show how to construct a simple c.i.s. around the level, having the above affine structure. The cotangent bundle of the desingularised level is used to perform the construction, and the c.i.s. obtained looks like the simplest one associated to the affine structure.\par
This method of construction is used to provide several examples of c.i.s. with different kinds of non-degenerate singularities.

\end{abstract}

\vskip 5mm

\noindent {\small{\it Subject Classification}: 53D05, 37J35,37G05
Hamiltonian system; Affine structure on a level; Normal forms}

    \section{ Introduction}
    \hspace{5mm}

Let $(M^{2n},\omega, F)$ be a non-degenerate integrable system.
This means that $(M^{2n},\omega)$ is a symplectic manifold and
$F=(f_{1},...,f_{n})$ is a proper moment map, which is
non-singular almost everywhere, and its singularities are of
Morse-Bott type.

 \vskip 2mm

 The $\Bbb R^{n}$-action generated by the
Hamiltonian vector fields $H_{f_{1}},..., H_{f_{n}}$ gives a
singular Lagrangian foliation on $(M^{2n},\omega)$. Any leaf is an
orbit of this actions. At the same time, any connected component
of $F^{-1}(c)$ is a level. We know that the regular levels are
$n$-dimensional tori and the singular levels are finite union of
several leaves.

\vskip 2mm

A semilocal classification of such integrable systems is still
open. It consists in  finding a complete system of invariants
describing symplectically  a neighborhood of a level. Some
approaches to solve this question has been currently made, see
\cite {BM}, \cite {Mir}, \cite {MiZ}, \cite {Mo}. As the local
description of non-degenerate singularities is given in terms of
products of elliptic, hyperbolic and focus-focus components, the
number of elliptic, hyperbolic and focus-focus components at each
point of the level will play an essential role in the
classification problem. \vskip 2mm

As in the regular case, the Hamiltonian vector fields
$H_{f_{1}},...,H_{f_{n}}$ endow  any leaf and any level with an
affine structure with singularities. In studying the semi-local
classification, we have seen that this affine structure gives strong
conditions on the set of invariants found. Our proposal in this
paper is to prove that this affine structure allows to the
construction of a completely integrable system around a given level
$L_{0}$ of $(M^{2n},\omega, F)$, such that the affine structure on
$L_{0}$ is the given one. This construction looks like the simplest
one with the given affine structure on $L_{0}$. \vskip 2mm As the
\lq \lq 1-jet \rq \rq  of the former completely integrable system
(c.i.s. from now on) and the one constructed in  this way coincide,
it takes sense to denote this c.i.s. as the linearized c.i.s. of the
initial one. \vskip 2mm

The process of construction of the linearized c.i.s. indicates us
a way to construct  c.i.s. (in fact thew will be linearized
completely integrable systems) with prescribed non-degenerate
singularities along a given  singular  level. Some constructions
are given.

\vskip 2mm
I would like to express my gratitude to Pierre Molino for interesting and fruitful conversations on this subject.

\vskip 15mm

 \section{Definition and basic properties.}
     \subsection{Local expressions} \label{subsection2.1}  Let
    $(M^{2n},\omega,f_{1},...,f_{n})$ be an integrable system,
    i.e., $(M^{2n},\omega)$ is a symplectic manifold. The
    functions $f_{1},...,f_{n}$ are Poisson commuting (first
    integrals of a given Hamiltonian system), such that
    $df_{1}\wedge \cdot \cdot \cdot \wedge df_{n}\ne 0$ on a dense
    subset of $M^{2n}$, and the moment map $F:M^{2n}
    \longrightarrow \Bbb R^{n}$, $F=(f_{1},...,f_{n})$, is proper.
    \vskip 2mm
    Such an integrable system is said to be non-degenerate
     if, in a neighborhood of each point $p_{0}\in M^{2n}$, there
    exist canonical coordinates  $(x_{1},y_{1},...,x_{n},y_{n})$, and $n$
    local functions $h_{1},...,h_{n}$, which have one of the
    following expressions:
    $$ h_{i}=y_{i} \qquad \text{(regular terms)}$$
    $$  h_{i}=(x_{i})^{2}+(y_{i})^{2} \qquad \text{(elliptic
    terms)} $$
    $$ h_{i}=x_{i}y_{i} \qquad \text{(hyperbolic terms)}$$
    $$ \begin{cases} h_{i}=x_{i}y_{i}+x_{i+1}y_{i+1} \\
    h_{i+1}=x_{i}y_{i+1}-y_{i}x_{i+1} \qquad \text{(focus-focus
    terms)} \end{cases} $$
    \noindent such that: \begin{enumerate} \item $f_{1},...,f_{n}$
    Poisson commute with $h_{1},...,h_{n}$.
    \item $\{j_{p_{0}}^{2}f_{1},...,j_{p_{0}}^{2}f_{n} \} $ and
    $\{j_{p_{0}}^{2}h_{1},...,j_{p_{0}}^{2}h_{n} \}$ generate the
    same space of $2$-jets at $p_{0}$. \end{enumerate}
    \vskip 5mm
     This notion of non-degeneracy implies obvious conditions on
     the space of $2$-jets generated at each point by
     $f_{1},...,f_{n}$. Conversely, assuming these infinitesimal
     conditions, the existence of such adapted coordinates is an
     important result, due to H.Eliasson \cite{E}. The demonstration
     has been completed by E.Miranda \cite{Mir} and E.Miranda \& V.N.San \cite{MirSan}
     . Adapted coordinates will be referred to as Eliasson coordinates, or
     simply E-coordinates.

\vskip 2mm

 Let $\cal {L}$ be the Lagrangian singular foliation
associated to this c.i.s., i.e. the leaves of $\cal{L}$ are the
orbits of the $\Bbb R^{n}$-action generated by
$H_{f_{1}},...,H_{f_{n}}$. \vskip 2mm

From a practical point of view, one can locally regard at
$(M^{2n},\omega, \cal {L})$ as  $(\Bbb
R^{2n},\omega_{0},{\cal{L}}_{0})$, where $\omega_{0}$ is the
standard symplectic two-form on $\Bbb R^{2n}$ and ${\cal {L}}_{0}$
is given by $dh_{i}=0$, $i=1,...,n.$

\

 Let $L_{0}$ be a singular level, and $p_{0}$ a point of $L_{0}$.
    By taking E-coordinates around $p_{0}$ we have the following
    characteristic numbers of the point $p_{0}$: the numbers $k_{e}$, $k_{h}$ and  $k_{f}$
     correspond to the number of elliptic, hyperbolic and focus-focus terms
    of the set $(h_{1},...,h_{n})$. The leaf through $p_{0}$ is
    $\Bbb T^{c}\times \Bbb R^{o}$, and the numbers $c$ and $o$ are
    called the degrees of closedness and openness of the leaf,
    respectively.
    \vskip 2mm

    Following Zung (see \cite{Z2}), the $5$-tuple
    $(k_{e},k_{h},k_{f},c,o)$ is called the leaf-type and
    $(k_{e},k_{h},k_{f})$ the Williamson type of $p_{0}$. In
    general one has $k_{e}+k_{h}+2k_{f}+c+o=n$.
    In \cite{Z2} it is proved that the three numbers $k_{e},
    k_{f}+c, k_{h}+k_{k}+o$, are invariants of each level. These
    numbers are known as the degrees of ellipticity, closedness
    and openness of the level.

    \vskip 5mm
     Summarizing, one can say that a singular level $L_{0}$  is a
     compact  $(n-k_{e})$-manifold, with self-intersections
     provided with a manifold structure of dimension less than
     $(n-k_{e})$. This level is
     endowed with a non-degenerate $\Bbb R^{n-k_{e}}$-action.
     Non-degenerate action means that the isotropy of this action
     at each point is linearizable.

\vskip 5mm

A we have proved (see \cite{Cu}) that a completely integrable
system, wit $k_{e}\ne 0$, is equivalent to a product of a standard
elliptic model (with degree of ellipticity equal to $k_{e}$) with
a c.i.s. with zero ellipticity degree, we will restrict our
attention along this paper to the case $k_{e}=0$.

\vskip 20mm

\subsection{Desingularized level}

We assume that on the singular level $L_{0}$, with singular affine
structure $\nabla _{0}$,  the degree of ellipticity vanishes. In
order to give a construction of a \lq \lq standard \rq \rq c.i.s.,
such that $L_{0}$ is a singular level, and the induced singular
affine structure on it coincides with $\nabla _{0}$, we start by
giving the construction of the  so called desingularized level:
\vskip 2mm As $k_{e}=0$ on $L_{0}$,  $\text{dim.} L_{0}=n$. Let
$\psi$ be a differentiable embedding of  $M^{2n}$ in an euclidean
space $\Bbb R^{l}$. Let $L_{0}^{r}\subset L_{0}$  be the subset of
regular points of $L_{0}$, i.e., a point $p\in L_{0}$ lies in
$L_{0}^{r}$ if and only if the rank of $dF$ is equal to $n$ at
$p$. Note that $L_{0}^{r}$ is not, in general, a connected
submanifold. \vskip 2mm We consider the following embedding of
$L_{0}^{r}$
$$
\varphi :L_{0}^{r} \hookrightarrow M^{2n}\times G_{n}(\Bbb R^{l})
\quad , \quad p \mapsto (p,[T_{\psi (p)}(\psi (L_{0}))]),
$$
\noindent where $G_{n}(\Bbb R^{l})$ is the $n$ dimensional Grassmann
manifold of $\Bbb R^{l}$.

\vskip 5mm
 \noindent We define the desingularized level,
${\hat{L}}_{0}$ as the closure of $\varphi (L_{0}^{r})$.

\vskip 10mm

\begin{claim} \label{Claim 2.1}$ {\hat{ L}}_{0}$ is a $n$-dimensional submanifold
of $M^{2n}\times G_{n}(\Bbb R^{l})$
\end{claim}

\begin{proof}
Let $q$ be a point of $\varphi (L_{0}^{r})$. As $\varphi$ is an
embedding, we have a $n$-dimensional natural chart defined around
$q$. \vskip 2mm Let $q=(p,[V])$ be a point in
$\text{Cl}(\varphi(L_{0}^{r})) \setminus \varphi(L_{0}^{r})$.
Obviously $p\in L_{0}\setminus L_{0}^{r}$. We know that there is a
chart $U$ in $M^{2n}$, with coordinates
$(x_{1},y_{1},...,x_{n},y_{n})$, centered at $p$, and $L_{0}\cap
U$ is given by $h_{i}=0$, $i=1,...,n$, where the functions $h_{i}$
are in the form of   section \ref{subsection2.1}. Let us have a look at the
tangent space to $L_{0}$ at a point near to $p$: from
the expressions of $h_{i}$, we see
$$
L_{0}\cap U=\prod _{j=1}^{s} (L^{j}_{0} \cap U^{j}), $$ \noindent
where each $ L_{0}^{j} \cap U^{j}$ is of the form \vskip 2mm
\noindent a) $U^{j}=D_{2}$, with coordinates $(x_{j},y_{j})$,
$h_{j}=x_{j}$, and $L_{0}^{j}\cap U^{j}$ is given by $h_{j}=0$
(regular term). \vskip 2mm \noindent b) $U^{j}=D_{2}$, with
coordinates $(x_{j},y_{j})$, $h_{j}=x_{j}y_{j}$, and
$L_{0}^{j}\cap U^{j}$ is given by $h_{j}=0$ (hyperbolic term).
\vskip 2mm \noindent c) $U_{j}=D_{2}\times D_{2}$, with
coordinates $ (x_{j},y_{j},x_{j+1},y_{j+1})$,
$h_{j}=x_{j}y_{j}+x_{j+1}y_{j+1}$,
$h_{j+1}=x_{j}y_{j+1}-x_{j+1}y_{j}$, and $L_{0}^{j}\cap U^{j}$ is
given by $h_{j}=h_{j+1}=0$ (focus-focus term). \vskip 5mm So, the
tangent space at a point near to $p$ will be: \vskip 5mm \noindent
In the case a) $$[T_{(0,y_{j})}(L_{0}^{j}\cap
U^{j})]=[<\frac{\partial}{\partial y_{j}}>],$$ \noindent so
$$[T_{(0,0)}(L_{0}^{j}\cap U^{j})]=[<\frac{\partial}{\partial
y_{j}}>]$$ \vskip 5mm \noindent in the case b)
$$[T_{(0,y_{j})}(L_{0}^{j}\cap U^{j})]=[<\frac{\partial}{\partial
y_{j}}
>]$$ \noindent or $$[T_{(x_{j},0)}(L_{0}^{j}\cap U^{j})]=[<\frac{\partial }{\partial x_{j}}>],$$ \noindent  so
$[T_{(0,0)}(L_{0}^{j}\cap U^{j})]$  is either
$$[<\frac{\partial}{\partial y_{j}}>] \quad   \text{or } \quad [<\frac{\partial
}{\partial x_{j}}>]$$.

 \vskip 5mm
\noindent and in the case c)
$$[T_{(0,y_{j},0,y_{j+1})}(L_{0}^{j}\cap U^{j})]=[<\frac{\partial}{\partial y_{j}},\frac{\partial}{\partial y_{j+1}}>] ,
$$  $$ \text{or} \quad
[T_{(x_{j},0,x_{j+1},0)}(L_{0}^{j}\cap
U^{j})]=[<\frac{\partial}{\partial x_{j}},\frac{\partial}{\partial
x_{j+1}}
>],$$ \noindent so
$$[T_{(0,0,0,0)}(L_{0}^{j}\cap U^{j})]=[<\frac{\partial}{\partial y_{j}},\frac{\partial}{\partial y_{j+1}}>]
\quad \text {or}$$
$$
[T_{(0,0,0,0)}(L_{0}^{j}\cap U^{j})]=[<\frac{\partial}{\partial
x_{j}}, \frac{\partial}{\partial x_{j+1}}
>].
$$

\vskip 5mm

So, in fact,there are $2^{h+f}$ possibilities for the class of the
grassmannian $[V]$ at $p \in L_{0} \setminus L_{0}^{r}$. \vskip
5mm
 As the structure of the chart around any one of these points will
 be a product structure, it will be sufficient to show the
 differentiability in the hyperbolic and in the focus-focus case.
 Let us do it, for instance, in the focus-focus  cases: we can
 consider that $p=(0,0,0,0)$ and $[V]=[<\frac{\partial}{\partial x_{j}},\frac{\partial}{\partial x_{j+1}}>]$,
 then the coordinates we take around $(p,[V])$ are
 $(x_{j},x_{j+1})$.
 \vskip 2mm
 One can check easily that with these charts, and the previous
 charts around the regular points, we have a structure of
 $C^{\infty}$-manifold on ${\hat {L_{0}}}$.

 \end{proof}

 \vskip 10mm

 The map $j:{\hat{L}}_{0} \longrightarrow M^{2n} \quad , \quad
 (p,[V])\longmapsto p $ is differentiable, $j({\hat{L}}_{0})=L_{0}$,  and the preimage of any
 singular point on $L_{0}$ consists of $2^{h+f}$ points on
$ {\hat{L}}_{0}$. This map is a Lagrangian immersion in
$(M^{2n},\omega)$

\vskip 15mm
\section{The singular cotangent model}

\subsection{Affine structure on the desingularized level}

The Poisson action of the local $F$-basic functions defines a
singular affine structure $\nabla _{0}$ on $L_{0}$. On each leaf
of the level, the affine structure is defined by considering the
infinitesimal generators of the $\Bbb R^{n}$-action as parallel
vector fields. This $\Bbb R^{n}$-action can be lifted in a natural
way to ${\hat{L}}_{0}$. Note that this lift is possible because
the map $j:{\hat{L}}_{0} \longrightarrow M^{2n}$ is an immersion.
As the affine structure on $L_{0}$ is provided by the Hamiltonian
vector fields $H_{f_{1}}\vert_{L_{0}}=X_{1},...,H_{f_{n}}\vert
_{L_{0}}=X_{n}$ it seems natural to write as ${\hat{X_{i}}}$,
$i=1,...,n$, the vector fields on ${\hat {L}}_{0}$ induced through
the immersion $j$, and ${\hat{\nabla}}_{0}$ this affine structure.

\vskip 4mm By using the affine structure on ${\hat {L}}_{0}$ we
are going to construct a completely integrable system on
$(T^{*}{\hat {L}}_{0}, {\hat {\omega}}_{0})$, where ${\hat
{\omega}}_{0}$ is the standard symplectic two-form on $T^{*}{\hat
{L}}_{0}$, and such that the affine structure on   ${\hat
{L}}_{0}$, induced by this c.i.s. is ${\hat {\nabla}} _{0}$.
\vskip 2mm As a last step, by a standard process of gluing, by
using natural local identifications, we will obtain a c.i.s., such
that $L_{0}$ is one level, and  the affine structure on $L_{0}$
will be $\nabla _{0}$.

\vskip 15mm
\subsection{A completely integrable system on
$T^{*}{\hat{L}}_{0}$} \label{subsection 3.2}

We define $n$ differentiable functions $g_{1}$,...,$g_{n}$  on
$T^{*}{\hat{L}}_{0}$ by $$ g_{i}(
(p,[V]),w):=<{\hat{X}}_{i}(p,[V]),w>
$$
Our proposal is to see that
$(T^{*}{\hat{L}}_{0},{\hat{\omega}}_{0},(g_{1},...,g_{n}))$ is
completely integrable. To do it, we need only to prove that $\{
g_{i},g_{j}\} \vert _{0} =0$. By a continuity argument, it  will
be sufficient to prove it at the points $((p,[V]),-)\in {\hat
{L}}_{0}$, where $p$ is a regular point. We can take Eliasson
coordinates $(x_{1},y_{1},...,x_{n},y_{n})$ around the point $p$,
and as any basic function only depends on $(y_{1},...,y_{n})$ in a
neighborhood of $p$, the expression of $H_{f_{i}}$ in this
neighborhood will be of the form $\sum_{j=1}^{n} {\frac {\partial
f_{i}}{\partial y_{j}} }(0,...,0)\frac {\partial}{\partial
x_{j}}$, i.e.  is a vector fields with constant coefficients
(which is obvious because the vector field is affine parallel).
\vskip 2mm Let $\alpha$ be the Liouville form on $T^{*}{\hat
{L}}_{0}$, in this neighborhood $\alpha =\sum _{i=1}^{n}
y_{i}dx_{i}$, $g_{i}=\sum _{k=1}^{n} y_{k} {\frac {\partial
f_{i}}{\partial y_{k}}}(0,...,0)$, so $dg_{i}=\sum_{k=1}^{n}{\frac
{\partial f_{i}}{\partial y_{k}}}(0,....,0) dy_{k}$, and obviously
$$
\Lambda ^{0} (dg_{i},dg_{j})=0.$$ \vskip 5mm

Once we know that
$(T^{*}{\hat{L}}_{0},{\hat{\omega}}_{0},(g_{1},...,g_{n}))$ is
completely integrable, we remark that, in general, is not proper:
let us assume, for instance that $\text{dim.}L_{0}=1$, and let $p$
be an hyperbolic point, we can write $\omega_{0}=dx\wedge dy$, and
$h=xy$, then $x=0$ is a leaf. \vskip 2mm In order to have a proper
completely integrable system around $L_{0}$we have to define a
gluing between points on ${\hat {L}}_{0}$ which are projected, via
$j$, to the same singular point on $L_{0}$. As the singular
composition of each point is, in fact, a product of hyperbolic and
focus-focus components, we need show how this gluing is done in
hyperbolic and focus-focus cases.

\vskip 5mm \noindent {\underbar {Hyperbolic gluing:}}
 Let
$(p,[V_{1}]), (p,[V_{2}])\in {\hat {L}}_{0}$, where $p \in L_{0}$ is
a purely hyperbolic point (degree of hyperbolicity equal to one).
Then in Eliasson coordinates $(x,y)$ around the point $p$, we may
assume $[V_{1}]=[<\frac {\partial }{\partial x}>]$ and
$[V_{2}]=[<\frac{\partial}{\partial y}>]$. \vskip 2mm We take
canonical coordinates in $T^{*}{\hat {L}}_{0}$ in two neighborhoods
$U_{1}$ of $(p,[V_{1}])$,  $U_{2}$ of $(p,[V_{2}])$, and we denote
them by $(x,y)$ and $(X,Y)$ respectively. Any basic function in
these neighborhoods depends on $xy$ and $XY$ respectively, and the
canonical symplectic two form is $dx\wedge dy$ and $dX\wedge dY$
respectively. \vskip 2mm The symplectomorphism from $U_{1}$ onto
$U_{2}$, we are searching for, is  expressed in these  coordinates
by $X=-y, Y=x$.

\vskip 5mm \noindent {\underbar{Focus-focus gluing:}}We can
proceed in a similar way for  two points $(p,[V_{1}])$,
$(p,[V_{2}])$, in the preimage of a focus-focus point $p\in
L_{0}$.  Regarding to Claim \ref{Claim 2.1} one can consider
$$[V_{1}]=[<\frac{\partial }{\partial x_{1}},\frac{\partial }{\partial
x_{2}}>] \quad , \quad [V_{2}]=[<\frac{\partial }{\partial
y_{1}},\frac{\partial }{\partial y_{2}}>].$$ We take canonical
coordinates in $T^{*}{\hat {L}}_{0}$  in two neighborhoods $U_{1}$
of $(p,[V_{1}])$  and  $U_{2}$ of $(p,[V_{2}])$, and we denote
them by $(x_{1},x_{2},y_{1},y_{2})$ and
$(X_{1},X_{2},Y_{1},Y_{2})$ respectively. Any basic function in
these neighborhoods depends on $x_{1}y_{1}+x_{2}y_{2} ,
-y_{1}x_{2}+x_{1}y_{2}$  and $X_{1}Y_{1}+X_{2}Y_{2},
-Y_{1}X_{2}+X_{1}Y_{2}$, respectively. The symplectomorphism from
$U_{1}$ to $U_{2}$ we need is given by: $X_1=-y_{1}$, $X_2=-Y_2$,
$Y_1
=X_1$, $Y_2=-X_2$.

\vskip 15mm

\subsection{The singular cotangent model}
Once we have shown  how  to identify the neighborhoods of  the
points in the same preimage we get a germ of $2n$-dimensional
 symplectic manifold  $(N, \omega _{0})$ containing $L_{0}$ as a  singular Lagrangian
 submanifold. Now we check that the  functions $g_{i}$, $i=1,...,n$, can be
 projected to $N$. To do it, we must see that the functions
 $g_{i}$ remain invariant under the above defined identifications
 in the hyperbolic and the focus-focus cases.
 \vskip 2mm
 In the hyperbolic case, we see that in the above defined
 neighborhood $U_{1}$, any function $g_{i}$ is of the form
 $y\frac {\partial h_{i}}{\partial y}(0)$, and a similar
 expression in $U_{2}$. Having in mind the symplectomorphism
 $Y=-x, X=y$, one sees that $h_{i}$ is preserved.
 \vskip 2mm
 For the focus-focus identification the proof is similar.

 \vskip 5mm

 It is quite obvious from the given construction of $(N, \omega
 _{0},(g_{1},...,g_{n}))$that the singular affine structure on
 $L_{0}$ is the previous one.

\vskip 15mm

\subsection{Construction of c.i.s. with prescribed singularities}

As a sort of application of the above considerations, let us point
out how to give some  c.i.s. with prescribed singularities, around a
singular level. \vskip 2mm The intrinsic geometry of the level, i.e.
the number and kind of singular points, is obviously related with
the kind of singularities along its  singular points. \vskip 5mm

We show how to obtain a c.i.s. around a $2$-dimensional singular
level with a focus-focus point and one circle of hyperbolic points:
\vskip 2mm Let us consider the $2$-sphere $S^{2}$. This manifold
$S^{2}$ will play the role of $\hat{L}_{0}$ of the above sections.
We shall consider $T^{*}S^{2}$ and two vector fields, with
singularities, which will be used to define two functions on
$T^{*}S^{2}$. These functions have singularities at several points,
and  by furnishing the  gluings at the corresponding singular
points, we  will get the c.i.s. around the singular level. \vskip
2mm Let $\theta$ (longitude) and $\varphi$ (latitude) be polar
coordinates on $S^{2}$.The vector fields to consider are:
$X=\frac{\partial}{\partial \theta}$ and $Y=h(\varphi)\frac
{\partial}{\partial \varphi}$. We have to give a \lq \lq good\rq\rq
expression for $h(\varphi)$, obviously we take
$h(-\frac{\pi}{2})=0$, $h(\frac{\pi}{2})=0$, and
$h(-\frac{\pi}{4})=h(\frac{\pi}{4})=0$. We consider four open
subsets of $S^{2}$
$$ U_{1}=\{(\theta,\varphi)\quad \vert \quad -\frac{\pi}{2}\le \varphi
<-\frac{\pi}{2}+\varepsilon \} $$
$$
V_{1}=\{(\theta,\varphi)\quad \vert \quad
-\frac{\pi}{4}-\varepsilon <\varphi < \frac{\pi}{4}+\varepsilon \}
$$
$$V_{2}=\{(\theta,\varphi)\quad \vert \quad \frac{\pi}{4}-\varepsilon <\varphi
<\frac{\pi}{4}+\varepsilon \}
$$
$$
U_{2}=\{ (\theta,\varphi)\quad \vert \quad
\frac{\pi}{2}-\varepsilon <\varphi \le \frac{\pi}{2} \}
$$
\noindent where  $\varepsilon <\frac{\pi}{8}$.

\vskip 5mm On $U_{1}$ and $U_{2}$we can take as coordinates the
first two cartesian coordinates $(x_{1},x_{2})$, and consider the
 vector fields: \vskip 2mm on $U_{1}$, $x_{1}\frac
{\partial}{\partial x_{1}}+x_{2}\frac{\partial}{\partial
x_{2}}=-\frac{\cos \varphi}{\sin \varphi}\cdot
\frac{\partial}{\partial \varphi}$ and
 $-x_{2}\frac{\partial}{\partial
x_{1}}+x_{1}\frac{\partial}{\partial
x_{2}}=\frac{\partial}{\partial \theta}$. \vskip 4mm

on $U_{2}$, $-x_{1}\frac {\partial}{\partial
x_{1}}-x_{2}\frac{\partial}{\partial x_{2}}=\frac{\cos
\varphi}{\sin \varphi}\cdot \frac{\partial}{\partial \varphi}$ and
$-x_{2}\frac{\partial}{\partial
x_{1}}+x_{1}\frac{\partial}{\partial
x_{2}}=\frac{\partial}{\partial \theta}$.

\vskip 5mm On $V_{1}$ and $V_{2}$ we can take $(\theta , \varphi)$
as coordinates and the following vector fields: \vskip 2mm
 on $V_{1}$, $-(\varphi
+\frac{\pi}{4})\frac{\partial}{\partial \varphi}$ and
$\frac{\partial}{\partial \theta}$. \vskip 2mm  on
$V_{2}$, $(\varphi-\frac{\pi}{4})\frac{\partial}{\partial
\varphi}$ and $\frac{\partial}{\partial \theta}$. \vskip 5mm Now
we see that the function $h(\varphi)$ we are searching for  can be
a differentiable function of $\varphi$, such that its values in
$U_{1},V_{1},U_{2},V_{2}$ are the above prescribed and not
vanishing on $S^{2}\setminus (U_{1}\cup V_{1}\cup U_{2} \cup
V_{2})$. \vskip 10mm

Finally, we consider on $T^{*}S^{2}$ the pair of functions $f, g$
associated to the vector fields $X,Y$, defined as follows: for any
point $(z,w)\in T^{*}S^{2}$, $f(z,w):=<X(z),w>$,
$g(z,w):=<Y(z),w>$. The c.i.s. we are searching for is obtained by
a process of gluing from $(T^{*}S^{2},\omega _{0},(f,g))$. This
gluing can be easily established by defining two
symplectomorphisms: one of them is a local symplectomorphism
between $(T^{*}U_{1}, (x_{1}=0,x_{2}=0,0,0))$ and
$(T^{*}U_{2},(X_{1}=0,X_{2}=0,0,0))$. We recall (see hyperbolic
gluing in section\ref{subsection 3.2}) that the mapping is given
by
$$
(x_{1},x_{2},y_{1},y_{2})\longmapsto
(X_{1}=-y_{1},X_{2}=-y_{2},Y_{1}=x_{1},Y_{2}=x_{2}).$$

\vskip 2mm The other symplectomorphism we need to conclude the
gluing is a semi-local symplectomorphism from
$(T^{*}V_{1},(\theta,\varphi=-\frac{\pi}{4},0,0))$ in
$T^{*}V_{2},(\theta,\varphi=\frac{\pi}{4},0,0))$, described as
follows: on $T^{*}V_{1}$ we take $(\theta,
\varphi +\frac{\pi}{4}, \Theta,\Phi)$ as canonical coordinates. In
the same form, we take $({\bar{\theta}}=\theta,{\bar{\varphi}}
-\frac{\pi}{4}=\varphi-\frac{\pi}{4},{\bar{\Theta}}=\Theta,{\bar{\Phi}})$
as canonical coordinates in $T^{*}V_{2}$. The symplectomorphism we
need to define the gluing is
$$
(\theta,\varphi +\frac{\pi}{4},\Theta,\Phi) \longmapsto
({\bar{\theta}}=\theta,{\bar{\varphi}}
-\frac{\pi}{4}=\Phi,{\bar{\Theta}}=\Theta,{\bar{\Phi}}=-(\varphi +
\frac{\pi}{4})).
$$

\vskip 10mm

Thus, the quotient of a germ of neighborhood of $S^{2}$ in
$T^{*}S^{2}$, by using these identifications provides us a germ of
completely integrable system around a level having one singular
point of focus-focus type and one circle of hyperbolic points.

\vskip 5mm One sees from this construction that the same arguments
can serve to give c.i.s. with several points of focus-focus type;
it should be necessary to use different copies of $S^{2}$ and
define a gluing by using the poles alternatively. The above
construction suggests different ways of having circles of
hyperbolic points in the level.

\vskip 40mm
 Departament d'\`{A}lgebra i Geometria, Universitat de
Barcelona (Spain)\par E-mail address:carloscurrasbosch@ub.edu

\end{document}